\documentclass[12pt,a4paper]{article}
\usepackage{amsmath,amssymb,amsthm, bm}
\usepackage{enumerate}
\usepackage[left=2.4cm,top=3.5cm,right=2.4cm,bottom=3.5cm]{geometry}
\usepackage{color}
\newtheorem{theorem}{Theorem}[section]
\newtheorem{corollary}[theorem]{Corollary}

\newtheorem{remark}[theorem]{Remark}

\newtheorem{example}[theorem]{Example}
\newtheorem{definition}[theorem]{Definition}
\numberwithin{equation}{section}
\usepackage{enumerate}
\newcommand{\Cm}{{\ensuremath{\mathbb{C}^{m\times n}}}}
\newcommand{\Cn}{{\ensuremath{\mathbb{C}^{n\times m}}}}
\newcommand{\Cnn}{{\ensuremath{\mathbb{C}^{n\times n}}}}

\newcommand{\Ra}{{\ensuremath{\cal R}}}
\newcommand{\Nu}{{\ensuremath{\cal N}}}

\newcommand{\core}{\mathrel{\text{\textcircled{$\#$}}}}
\newcommand{\ra}{{\ensuremath{\rm rk}}}
\newcommand{\ind}{{\ensuremath{\rm Ind}}}

\newcommand{\odagger}{\mathrel{\text{\textcircled{$\dagger$}}}}
\newcommand{\weak}{\mathrel{\text{\textcircled{w}}}}

\newcommand{\rank}{{\ensuremath{\rm rk}}}

\begin{document}
\author{D.E. Ferreyra\thanks{Universidad Nacional de R\'io Cuarto, CONICET, FCEFQyN, RN 36 KM 601, 5800 R\'io Cuarto, C\'ordoba, Argentina. E-mail: \texttt{deferreyra@exa.unrc.edu.ar}} ~and
Saroj B. Malik\thanks{School of Liberal Studies, Ambedkar University, Kashmere Gate, Delhi, India. E-mail:  \texttt{saroj.malik@gmail.com}}}

\title{The $m$-weak core inverse}
\date{}
\maketitle

\begin{abstract}
Since the day the core inverse has been known in a paper of Bakasarly and Trenkler, it has been widely researched. 
So far, there are four generalizations of this inverse for the case of matrices of an arbitrary index, namely, the BT inverse, the DMP  inverse, the core-EP inverse  and the WC inverse. In this paper we introduce a new type of generalized inverse for a matrix of arbitrary index to be  called $m$-weak core inverse which generalizes the core-EP inverse,  the WC inverse, and therefore the core inverse. We study several properties and characterizations of the $m$-weak core inverse  by using matrix decompositions.
\end{abstract}

AMS Classification: 15A09, 15A24.

\textrm{Keywords}: Drazin inverse, $m$-weak core inverse, $m$-weak group inverse, DMP inverse, core-EP inverse, WC inverse, core inverse.

\section{Introduction}

Generalized matrix inverses  are important objects in matrix theory. They were defined to find solutions of systems of algebraic equations to begin with and in due course of time they  have proved to be useful tools not only in solving matrix equations but also in various applications such as networks, coding theory, environmental chemistry, electrical engineering and statistics to cite a few \cite{BeGr}.

The earliest generalized inverses known are the Moore-Penrose inverse and the Drazin inverse (or more specifically the group inverse). Using the Moore-Penrose inverse and the group inverse, Baksalary and Trenkler  \cite{BaTr} introduce the core inverse of a square matrix of index at most one. A large number of research papers can be found in the literature on this generalized inverse for the last decade \cite{FeMa, FeMa2, FeMa3, FeLeTh2, FeLeTh3} and several others to name. The core inverse has been generalized further to give rise to several new generalized inverses namely, the BT inverse \cite{BaTr2}, the DMP  inverse \cite{MaTh}, the core-EP inverse \cite{MoPr} and the WC inverse \cite{FeLePrTh}.  The recent generalized inverses show new research trends that cover both theoretical and applied aspects. Some of the interesting applications of these generalized inverses can be found, for example, in \cite{FeMa4, MoStZh, StSeBeSaMoSeLa, ZhChStKaMa}. Also, they play an important role in the study of matrix partial orders, as we can see in \cite{FeMa, FeMa2, MiBhMa, ZhPa}.

Our main objective in this paper is to introduce a new type of generalized inverse  of a matrix of an arbitrary index called $m$-weak core inverse which  unifies two  generalizations of the inverse core namely, the core-EP inverse and the WC inverse.  We study several properties and characterizations of the $m$-weak core and present some of its main representations  and canonical forms.

\section{Notations and terminology}

We first give the notations to be used in this paper. We also record the  definitions of some known generalized inverses that are useful in this work.

We denote the set of all $m \times n$ complex matrices by $\Cm$. For $A\in\Cm$, the symbols $A^{\ast},$ $A^{-1}$, $\ra(A)$, $\Nu(A)$, and  $\Ra(A)$
will stand for the conjugate transpose,
the inverse (when $m=n$), the rank, the null space, and the column space of $A$, respectively. Moreover, $I_n$ will refer to the $n \times n$ identity matrix.

A matrix $X \in \Cn$ that
satisfies the equality $AXA =A$ is called an {\it inner inverse}, and the set of all inner inverses of $A$ is denoted by $A\{1\}$. While a matrix $X \in \Cn$ that satisfies the equality $XAX =X$ is called an {\it outer inverse} of $A$.

For $A\in\Cm$, the {\it Moore-Penrose inverse} of $A$ is the unique matrix $A^\dag \in \Cn$ satisfying the following four equations \cite{BeGr}
$$ AA^\dag A=A, \quad A^\dag A A^\dag=A^\dag,
\quad (A A^\dag)^*=A A^\dag, \quad   \text{and} \quad (A^\dag A)^*=A^\dag A.$$
The Moore-Penrose inverse can be used to represent  the orthogonal projectors $P_A:=AA^{\dag}$ and $Q_A:=A^{\dag}A$ onto $\Ra(A)$ and $\Ra(A^*)$, respectively.
When $\Ra(A)=\Ra(A^*)$ (or equivalently $P_A=Q_A$) we say that $A$ is an EP matrix.

Let $A\in \Cnn$. The smallest nonnegative integer $k$ for which ${\cal R}(A^k) = {\cal R}(A^{k+1})$ is called the {\it index} of $A$ and is denoted by $\text{Ind}(A)$. It is well-known that if $A$ is EP then $A$ is of index at most one. 

We recall that the {\it Drazin inverse} of $A\in \Cnn$ is the unique matrix
$A^d\in \Cnn$ satisfying the following three equations \cite{BeGr}
$$A^dAA^d=A^d, \quad AA^d=A^dA, \quad   \text{and} \quad A^d  A^{k+1}=A^k, $$
where $k=\ind(A)$.
If $k\le 1$, the  Drazin inverse of $A$ is called the  {\it group inverse} of $A$ and denoted by $A^\#$.

In 2010, Baksalary and Trenkler \cite{BaTr} introduced the core inverse of a square matrix.
They proved that  the core inverse of a matrix $A \in \Cnn$ is  the unique matrix $A^{\core} \in \Cnn$ satisfying the conditions
\begin{equation}\label{core inverse}
AA^{\core}=P_A \quad \text{and} \quad \Ra(A^{\core})\subseteq \Ra(A).
\end{equation}
It is well-known that the core inverse of $A$ exists if and only if $\ind(A) \le 1$. In this case, the unique solution satisfying \eqref{core inverse} is $A^{\core}=A^\# A A^\dag$.

In 2014, two of the most known generalizations of the core inverse were introduced for
$n \times n$ complex matrices of arbitrary index $k$, namely, the DMP inverse and the core-EP inverse.
More precisely, Malik and Thome \cite{MaTh} introduced the {\it DMP inverse} of $A \in \Cnn$ as the unique matrix $A^{d,\dag} \in \Cnn$ satisfying
\begin{equation}\label{def dmp inverse}
A^{d,\dag}AA^{d,\dag}=A^{d,\dag},\qquad
A^{d,\dag}A=A^d A, \quad \text{and} \quad A^kA^{d,\dag}= A^k A^{\dag}.
\end{equation} Moreover, it was
proved that $A^{d,\dag}=A^dAA^\dag$.

Manjunatha Prasad and Mohana \cite{MoPr} defined the {\it core-EP inverse} of $A \in \Cnn$ as the  unique matrix $A^{\odagger} \in \Cnn$ satisfying
\begin{equation}\label{def core EP}
A^{\odagger}AA^{\odagger}=A^{\odagger}\qquad \text{and} \qquad
\Ra(A^{\odagger})=\Ra((A^{\odagger})^*)=\Ra(A^k).
\end{equation}
As was proved in \cite[Corollary 3.8]{FeLeTh3}, the core-EP inverse can be represented by the formulae
\begin{equation}\label{core-EP inverse}
A^{\odagger}=A^d P_{A^k}.
\end{equation}

By using the core-EP inverse of a matrix, Wang and Chen \cite{WaCh} introduced the {\it WG inverse}  of a matrix $A \in \Cnn$  as the unique matrix $A^{\weak}\in \Cnn$ satisfying
\begin{equation} \label{def wg inverse}
A(A^{\weak})^2=A^{\weak} \quad  \text{and} \quad AA^{\weak}=A^{\odagger}A.
\end{equation}
The unique solution of the above system is given by
\begin{equation}\label{WG inverse}
A^{\weak}=(A^{\odagger})^2 A.
\end{equation}
If $\ind(A)\le 1$, the WG inverse and the group inverse coincide.
In \cite{FeThOr} the authors extended the notion of the WG inverse to rectangular matrices.

In \cite{ZhChZh}, Zhou et al. proposed the $m$-weak group in ring with involution and gave some characterizations of it. For $A\in\Cnn$  a matrix of index $k$, the $m$-weak group inverse is the unique matrix $A^{\weak_m}=X\in \Cnn$ satisfying the equations
\begin{equation}
XA^{k+1}=A^k, \quad AX^2=X, \quad (A^*)^kA^{m+1}X=(A^*)^kA^m, \quad m\in \mathbb{N}.
\end{equation}

Recently, in \cite{JiZu} it was proved that the $m$-weak group inverse of $A$ can be characterized by the following matrix system involving the core-EP inverse
\begin{equation}\label{m-WG characterization}
AX^2=X, \quad AX=(A^{\odagger})^m A^m.
\end{equation}
Moreover, the authors proved that the $m$-weak group inverse can be expressed in term of the  core-EP inverse as
\begin{equation}\label{m-WG}
A^{\weak_m}=(A^{\odagger})^{m+1} A^m.
\end{equation}
\begin{remark}\label{rem 1}
From \eqref{WG inverse} and \eqref{m-WG} it is clear that if $m=1$, the $m$-weak group inverse reduces to the WG inverse, i.e., $A^{\weak_1}=A^{\weak}$.
When $m=2$, the $m$-weak group inverse coincides with the GG inverse studied recently by Ferreyra and Malik in \cite{FeMa}, that is, $A^{\weak_2}=(A^{\odagger})^3 A^2$. Moreover, if $m\ge k=\ind(A)$,  $A^{\weak_m}=A^d$. Therefore, the $m$-weak group inverse extends the notions of  WG inverse, GG inverse and Drazin inverse.
\end{remark}

Recently, Ferreyra et al. \cite{FeLePrTh}  introduced another generalization of the core inverse by using the WG inverse, namely the  {\it WC inverse} as the matrix
\begin{equation}\label{WC inverse}
A^{\weak,\dag}=A^{\weak} P_A.
\end{equation}

This paper is organized as follows. Section 3 comprises of some preliminaries. Section 4 introduces the $m$-weak core inverse and  some of its main properties by using the core-EP decomposition. In Section 5, we derive some more properties of this new generalized inverse. Section 6 is devoted to the study another canonical form of $m$-weak core inverses. Section 7 offers the relationship between the $m$-weak core inverse and other generalized inverses.

\section{Preliminaries}

As proved in \cite{Wang}, every matrix $A\in \Cnn$ of index $k$ can be written in its core-EP decomposition \begin{equation} \label{core EP decomposition}
A= U\left[\begin{array}{cc}
T & S\\
0 & N
\end{array}\right]U^*,
\end{equation}
where $U\in \Cnn$ is unitary, $T\in \mathbb{C}^{t\times t}$ is a nonsingular matrix with $t:=\text{rk}(T)=\text{rk}(A^k)$, and $N\in \mathbb{C}^{(n-t)\times (n-t)}$ is nilpotent of index $k$.

Notice that if $k=0$, $N$ and $S$ are absent in \eqref{core EP decomposition}, and $A = T$ (with $U=I_n$). Henceforth, we can assume $\text{Ind}(A)=k \geq 1$ when the core EP decomposition is applied.

If  $A\in \Cnn$ is written as \eqref{core EP decomposition}, the core-EP inverse of $A$ is
\begin{equation}\label{core EP representation}
A^{\odagger} = U\left[\begin{array}{cc}
T^{-1} & 0 \\
0 & 0
\end{array}\right]U^*.
\end{equation}

\begin{remark}\label{GM} If $A$ is of the form \eqref{core EP decomposition} then $k=1$ if and only if $N=0$.
\end{remark}

\begin{remark} \label{A^m decomposition}
Let $A\in \Cnn$ be  written as in (\ref{core EP decomposition}) and $\ell \in \mathbb{N}$. Then
\begin{equation} \label{core EP Am}
A^\ell = U\left[\begin{array}{cc}
T^\ell & \tilde{T}_{\ell}\\
0 & N^{\ell}
\end{array}\right]U^*, \quad \tilde{T}_{\ell}:=\sum\limits_{i=0}^{{\ell}-1} T^{i} S N^{{\ell}-1-i},
\end{equation}
 is the core-EP decomposition of $A^{\ell}$. In particular, by \cite[Theorem 3.9]{FeLeTh3} we know that the Moore-Penrose inverse of $A^{\ell}$ is given by
 \begin{equation} \label{MP of Am}
(A^{\ell})^\dagger = U\left[\begin{array}{cc}
(T^{\ell})^*\Delta_{\ell} & -(T^{\ell})^*\Delta_{\ell} \tilde{T}_{\ell} (N^{\ell})^\dagger\\
\Omega_{\ell}^*\Delta_{\ell} & (N^{\ell})^\dagger-\Omega_{\ell}^*\Delta_{\ell} \tilde{T}_{\ell} (N^{\ell})^\dagger
\end{array}\right]U^*,
\end{equation} where
 \begin{equation*}
\Delta_{\ell}:=(T^{\ell} (T^{\ell})^*+ \Omega_{\ell}\Omega_{\ell}^*)^{-1}\quad \text{and} \quad \Omega_{\ell}:=\tilde{T}_{\ell}(I_{n-t}-Q_{N^{\ell}}).
\end{equation*}
\end{remark}

From \eqref{core EP Am} and \eqref{MP of Am} we derive the following expression for the orthogonal projector

\begin{equation}\label{P_Am and Q_Am}
P_{A^\ell} = U\left[\begin{array}{cc}
I_t &  0 \\
0 & P_{N^\ell}
\end{array}\right]U^*, \quad \ell\in \mathbb{N}.
\end{equation}

Moreover, as proved in \cite{FeLeTh1}, for each integer $\ell \ge k$  we have
\begin{equation}\label{projector}
P_{A^\ell}
= U\left[\begin{array}{cc}
I_t & 0 \\
0 & 0
\end{array}\right]U^*.
\end{equation}

\section{The $m$-weak core inverse}

In this section we introduce a generalized inverse of $A \in \Cnn$ by using the $m$-weak group inverse $A^{\weak_m}$ of $A$ and the orthogonal projector onto $\Ra(A^m)$, for $m\in \mathbb{N}$.

\begin{definition}\label{def m-weak core}
Let  $A\in \Cnn$ and $m\in \mathbb{N}$. The matrix $A^{\core_m}\in \Cnn$ satisfying
\begin{equation}\label{m-weak core}
A^{{\core}_m}=A^{\weak_m} P_{A^m},
\end{equation}
is called  the $m$-weak core inverse of $A$.
\end{definition}

As the $m$-weak group inverse of a square matrix always exists and is unique, it is seen from \eqref{m-weak core} that the $m$-weak core inverse exists for every square matrix and is unique.

From  \eqref{core-EP inverse},  \eqref{WC inverse} and Remark \ref{rem 1} we can say that the  $m$-weak core inverse is a generalization  of the WC inverse and  the core-EP inverse.

\begin{remark}\label{rem 2}
\begin{enumerate}[(i)]
\item   If $m=1$, then the $m$-weak core inverse coincides with the WC inverse, that is, $A^{\core_1}=A^{\weak,\dag}$.
\item  If $m\ge k$, then the $m$-weak core inverse coincides with the core-EP inverse, that is, $A^{\core_m}=A^{\odagger}$.
\item  If $A^m$ is EP then $A^{\core_m}=A^{\weak_m}$. In fact, as $A^m$ is EP then  $P_{A^m}=Q_{A^m}$. Thus, $A^{\core_m}=A^{\weak_m} P_{A^m}=(A^{\odagger})^{m+1}A^m Q_{A^m}=(A^{\odagger})^{m+1}A^m=A^{\weak_m}$.
\end{enumerate}
\end{remark}
Notice that, in general, the reciprocal of Remark \ref{rem 1} (iii) is false.
\begin{example} Let
$A= \left[\begin{array}{rrrrr}
1 & 0 & 0 & 1 & 0 \\
0 &  0 & 1 & 1 & 0  \\
0 &  0 &  0 & 1 & 0\\
0 & 0 & 0 & 0 & -1 \\
0 & 0 & 0 & 0 & 0
\end{array}\right]$.
It is easy to check that $\text{Ind}(A)=4$. Consider $m=1$. Thus, the $m$-weak group inverse and the $m$-weak core inverse coincides, that is, \[A^{\weak_1}=A^{\core_1}= \left[\begin{array}{rrrrr}
1 & 0 & 0 & 1 & 0 \\
0 &  0 & 0 & 0 & 0  \\
0 &  0 &  0 & 0 & 0\\
0 & 0 & 0 & 0 & 0\\
0 & 0 & 0 & 0 & 0
\end{array}\right].\]
 However, clearly $A$ is not EP.
\end{example}
\begin{remark}\label{rem 3} If $\ind(A)\le 1$, it is well-known that $A^{\odagger}=A^{\core}$. So,  from Remark \ref{rem 1} (ii), the $m$-weak core inverse coincides with the core inverse for all $m\in \mathbb{N}$ provided $\ind(A)\le 1$.
\end{remark}
Next we give some basic properties   of the $m$-weak core inverse.

\begin{theorem} \label{properties 1}
Let  $A\in \Cnn$ and $m\in \mathbb{N}$.  Then
\begin{enumerate}[{\rm(a)}]
\item  $A^{\core_m}=(A^{\odagger})^{m+1}A^m P_{A^m}$.
\item  $A^{\core_m}=(A^{\weak})^m A^{m-1}P_{A^m}$.
\item  $A A^{\core_m}=(A^{\odagger})^m A^m P_{A^m}=A^{\weak_{m-1}}AP_{A^m}$.
\item  $A^{\core_m} A=(A^{\odagger})^{m+1}A^m P_{A^m}A$.
\item  $A^{\core_m} A^m=A^{\weak_m}A^m=(A^{\odagger})^{m+1}A^{2m}$.
\item  $A^{\core_m}=YAA^{\core_m}$, where $YAY=Y$ and $\Ra(Y)=\Ra(A^k)$.
\end{enumerate}
\end{theorem}
\begin{proof}
(a) It directly follows from \eqref{m-WG} and \eqref{m-weak core}.\\
(b) From \eqref{WG inverse} we know that $A^{\weak}=(A^{\odagger})^2A$. Also, it is well-known that $A(A^{\odagger})^2=A^{\odagger}$. Thus, by applying induction on $m$, it is easy to see $(A^{\weak})^m=(A^{\odagger})^{m+1}A$. Now, the affirmation follows from (a). \\
(c) As $A^{{\core}_m}=A^{\weak_m} P_{A^m}$, the first equality in (c) follows  from \eqref{m-WG characterization}. The second equality is due to \eqref{m-WG}. \\
(d) It is a consequence  from (a).\\
(e) The first equality is an immediate consequence from Definition \ref{def m-weak core}. The second equality is due to \eqref{m-WG}.\\
(f)  Note that $A^{\core_m}=YAA^{\core_m}$ holds if and only if $\Ra(A^{\core_m})\subseteq \Nu(I_n-YA)=\Ra(YA)=\Ra(Y)=\Ra(A^k)$, because $Y$ is an outer inverse of $A$. Now, by Theorem \ref{properties 2} (h)    also we have $\Ra(A^{\core_m})=\Ra(A^k)$. Thus, the assertion follows.
\end{proof}

Now we identify a representation for the $m$-weak core inverse by using the core-EP decomposition.
\begin{theorem}\label{thm canonical m-weak core}
Let $A \in \Cnn$ be of the form (\ref{core EP decomposition}) and $m\in \mathbb{N}$. Then
\begin{equation}\label{canonical m-weak core}
A^{\core_m}=U\begin{bmatrix}
 T^{-1} & T^{-(m+1)}\tilde{T}_m P_{N^m} \\
  0 & 0
  \end{bmatrix}U^*.
\end{equation}
\end{theorem}
\begin{proof}
By Theorem \ref{properties 1} (a) we have $A^{\core_m}=(A^{\odagger})^{m+1}A^m P_{A^m}$. By using the canonical forms of $A^{\odagger}$, $A^m$, and $P_{A^m}$ given  in
\eqref{core EP representation}, \eqref{core EP Am}, and \eqref{P_Am and Q_Am}, respectively, we have
\begin{eqnarray*}
A^{\core_m} &=& U \begin{bmatrix}
 T^{-(m+1)} & 0 \\
  0 & 0
  \end{bmatrix}\begin{bmatrix}
 T^m & \tilde{T}_m \\
  0 & N^m
  \end{bmatrix}\left[\begin{array}{cc}
I_t &  0 \\
0 & P_{N^m}
\end{array}\right]U^* \\
&=&
U \begin{bmatrix}
 T^{-1} & T^{-(m+1)}\tilde{T}_m \\
  0 & 0
  \end{bmatrix}\left[\begin{array}{cc}
I_t &  0 \\
0 & P_{N^m}
\end{array}\right]U^* \\
  &=& U\begin{bmatrix}
 T^{-1} & T^{-(m+1)}\tilde{T}_m P_{N^m} \\
  0 & 0
  \end{bmatrix}U^*.
\end{eqnarray*}
\end{proof}

From Theorem \ref{thm canonical m-weak core}, it is possible to obtain more features/properties of the $m$-weak core inverse, some of which   coincide with those attributed to the core inverse.

\begin{theorem} \label{properties 2}
Let $A\in \Cnn$, $l\ge k=\ind(A)$, and $m\in \mathbb{N}$.
Then
\begin{enumerate}[{\rm(a)}]
\item $A^{\core_m}AA^{\core_m}=A^{\core_m}$, i.e., $A^{\core_m}$ is an outer inverse of $A$.
\item $A_1A^{\core_m}A_1=A_1$ and $A^{\core_m}A_1A^{\core_m}=A^{\core_m}$, where $A_1=AA^{\odagger}A$.
\item $A(A^{\core_m})^2=A^{\core_m}$.
\item $A^{\core_m} A^{l+1}=A^l$.
\item $A^{\core_m}=(A^d)^{m+1} P_{A^l} A^m P_{A^m}$.
\item $A^{\core_m}=A^l(A^{l+m+1})^\dag A^m  P_{A^m}$.
\item $\ra(A^{\core_m})=\ra(A^k)$.
\item $\Ra(A^{\core_m})=\Ra(A^k)$.
\item $\Nu(A^{\core_m})=\Nu((A^k)^*A^m P_{A^m})$.
\item $A^{\core_m} A^{m+1}=(A^{\odagger})^{m+1}A^{2m+1}$.
\end{enumerate}
\end{theorem}
\begin{proof}
The assertions (a)-(c) can be verified easily by using  \eqref{core EP decomposition} and \eqref{canonical m-weak core}. \\
(d) From \eqref{core EP Am} we have
\[A^l = U\left[\begin{array}{cc}
T^l & \tilde{T}_l \\
0 & 0
\end{array}\right]U^*, \quad \ell\ge k.\]
Thus, \eqref{canonical m-weak core} implies
\[A^{\core_m}A^{\ell+1}=(A^{\core_m}A)A^\ell=U\begin{bmatrix}
I_t & T^{-1}S+T^{-(m+1)}\tilde{T}_m P_{N^m} N\\
  0 & 0 \\
  \end{bmatrix}\left[\begin{array}{cc}
T^{l} & \tilde{T}_{l} \\
0 & 0
\end{array}\right]U^*=A^l.\]
(e) From \cite[Theorem 2.2]{FeLePrTh} we know \[A^d=U\begin{bmatrix}
 T^{-1} & T^{-(k+1)}\tilde{T}_k  \\
  0 & 0 \\
  \end{bmatrix}.\]
In consequence,  \eqref{projector} implies
\begin{equation}\label{m+1 dagger}
(A^d)^{m+1}P_{A^l}=U\begin{bmatrix}
 T^{-(m+1)} & T^{-(m+k+1)}\tilde{T}_k  \\
  0 & 0 \\
  \end{bmatrix}\begin{bmatrix}
 I_t & 0  \\
  0 & 0 \\
  \end{bmatrix}U^*=\begin{bmatrix}
 T^{-(m+1)} & 0  \\
  0 & 0 \\
  \end{bmatrix}=(A^{\odagger})^{m+1},
\end{equation}
where the last equality is due to \eqref{core EP representation}. \\
Thus, the statement follows from Theorem \ref{properties 1} (a). \\
(f) Since $m+l+1\ge k$, from (e) we obtain $A^{\core_m}=(A^d)^{m+1} P_{A^{m+l+1}} A^m P_{A^m}$. Thus, as $l\ge k$ we have
\begin{eqnarray*}
A^{\core_m} &=& (A^d)^{m+1} A^{m+l+1} (A^{m+l+1})^\dag A^m P_{A^m}\\
&=& (A^dA)^{m+1} A^l (A^{m+l+1})^\dag A^m P_{A^m} \\
&=& A^d A^{l+1} (A^{m+l+1})^\dag A^m P_{A^m}\\
&=& A^l (A^{m+l+1})^\dag A^m P_{A^m}.
\end{eqnarray*}
(g) From \eqref{core EP Am}, for $\ell=k$, we have
\begin{equation}\label{A^k}
A^k=U\begin{bmatrix}
 T^k & \tilde{T}_k  \\
  0 & 0 \\
  \end{bmatrix}.
\end{equation}
Now, from  \eqref{canonical m-weak core} and \eqref{A^k} it is easy to see that $\ra(A^{\core_m})=\ra(T)=\ra(A^k)$.\\
(h) By Theorem \ref{properties 1} (a) we have $\Ra(A^{\core_m})\subseteq \Ra(A^{\odagger})=\Ra(A^k)$, where the last equality is due to  \eqref{def core EP}. Now, (g) completes the statement.\\
(i) By Theorem \ref{properties 1} (a)  we  get
\begin{equation}\label{null}
\Nu (A^{\odagger} A^m P_{A^m})\subseteq \Nu((A^{\odagger})^{m+1} A^m P_{A^m})= \Nu(A^{\core_m}).
\end{equation}
Also, from \eqref{core EP representation} and \eqref{core EP Am} one can obtain $A^{\odagger} =  A^m (A^{\odagger})^{m+1}$. So, Theorem \ref{properties 1} (a) yields
\begin{equation}\label{rank}
\ra(A^{\core_m})\le \ra(A^{\odagger} A^m P_{A^m} )\le \ra((A^{\odagger})^{m+1} A^m P_{A^m})=\ra(A^{\core_m}).
\end{equation}
From \eqref{null} and \eqref{rank} we obtain $\Nu(A^{\core_m})=\Nu(A^{\odagger}A^mP_{A^m})$. Now, by \cite[Theorem 3.2]{FeLeTh3} we know that $\Nu (A^{\odagger})=\Nu ((A^k)^*)$. In consequence, 
\[x\in \Nu (A^{\core_m}) \Leftrightarrow A^mP_{A^m}x\in \Nu (A^{\odagger})=\Nu ((A^k)^*).\]
Therefore,
$x\in \Nu (A^{\core_m})$ if and only if  $ x\in \Nu ((A^k)^* A^mP_{A^m})$.
\\
(j) It is a direct consequence from Theorem \ref{properties 1} (a).
\end{proof}

Two important consequences  of  Theorems \ref{properties 1} and  \ref{properties 2} are  the following  characterizations of the $m$-weak core inverse.

\begin{theorem} \label{characterization 1} Let $A\in \Cnn$ and $m\in \mathbb{N}$. Then  the system of equations
\begin{equation}\label{system 1}
XAX=X, \quad AX=(A^{\odagger})^m A^m P_{A^m},\quad XA=(A^{\odagger})^{m+1}A^m P_{A^m}A,
\end{equation}
is consistent and its unique solution is the matrix $X=A^{\core_m}$.
\end{theorem}
\begin{proof}
\emph{Existence.} Let $X:=A^{\core_m}$. From Theorems \ref{properties 1} and  \ref{properties 2}  it is easy to see that $X$ satisfies  all the  equations in \eqref{system 1}. \\
\emph{Uniqueness.} We assume the $X_1$ and $X_2$ are two solutions of the system of equations in (\ref{system 1}). Then $X_2=X_2AX_2=X_1AX_2=X_1AX_1=X_1$.
 \end{proof}

\begin{remark} Note that the second and third  conditions in system \eqref{system 1} can be replaced  by $AX=A^{\weak_{m-1}} A P_{A^m}$ and $XA=A^{\weak_m}P_{A^m}A$, respectively.
\end{remark}

\begin{theorem}\label{characterization 2} Let $A\in \Cnn$ with $\ind(A)=k$, and $m\in \mathbb{N}$.  Then the system of equations
\begin{equation}\label{system 2}
AX=(A^{\odagger})^m A^m P_{A^m},\quad \Ra(X)\subseteq \Ra(A^k),
\end{equation}
is consistent and its unique solution is the matrix $X=A^{\core_m}$.
\end{theorem}
\begin{proof}
\emph{Existence.} Let $X:=A^{\core_m}$. Theorems \ref{properties 1} and  \ref{properties 2}  it is easy to see that $X$ satisfies  all the  equations in \eqref{system 2}.  \\
\emph{Uniqueness.}  Let $X_1$ and $X_2$ be two solutions of the system \eqref{system 2},
that is, \[AX_1=AX_2=(A^{\odagger})^m A^m P_{A^m}, \quad \Ra(X_1)\subseteq \Ra(A^k), \quad \Ra(X_2)\subseteq\Ra(A^k).\] 
Thus,   $\Ra(X_1-X_2)\subseteq \Nu(A)\subseteq \Nu(A^k)$ and
$\Ra(X_1-X_2)\subseteq \Ra(A^k).$ Therefore, $\Ra(X_1-X_2)\subseteq \Ra(A^k)\cap \Nu(A^k)=\{0\}$ because $A$ has index $k.$ Thus, $X_1=X_2.$
\end{proof}

\begin{remark} Note that the first condition in system \eqref{system 2} can be replaced  by $AX=A^{\weak_{m-1}} A P_{A^m}$.
\end{remark}

\begin{remark} Note that when $k=1$, system \eqref{system 2} reduces to \eqref{core inverse}, that is, we obtain the conditions that characterize the core inverse. In fact, as $m\ge k=1$, from \eqref{core EP representation}, \eqref{core EP Am}, and \eqref{P_Am and Q_Am}, slight calculations lead to
$AX=(A^{\odagger})^m A^m P_{A^m}=P_A$. Obviously, $\Ra(X)\subseteq \Ra(A)$.\\
When $m=1$, the equations in \eqref{system 2} reduces to the conditions obtained in \cite[Theorem 3.11]{FeLePrTh}. \\
While if $m\ge k$, system \eqref{system 2} reduces to $AX=P_{A^k}$ and 
$\Ra(X)\subseteq \Ra(A^k)$, which is a well-known characterization of the core-EP inverse obtained in \cite[Theorem 2.7]{FeLeTh1}.
\end{remark}

From above theorem we can deduce more necessary and sufficient conditions for a square matrix to be the $m$-weak core inverse. 

\begin{corollary}\label{characterization 3}
Let $A \in \Cnn$ with $\text{Ind}(A)=k$. Then the following statements are equivalent:
\begin{enumerate}[{\rm(a)}]
\item $X$ is the  $m$-weak core inverse $A^{\core_m}$ of $A$;
\item $AX=(A^{\odagger})^m A^m P_{A^m}$ and $\Ra(X)=\Ra(A^k)$;
\item $AX=(A^{\odagger})^m A^m P_{A^m}$ and $AX^2=X$.
\end{enumerate}
\end{corollary}
\begin{proof}
(a) $ \Rightarrow $ (b). It follows by applying  
Theorem \ref{properties 1} (c) and  Theorem \ref{properties 2} (h).  \\
(b) $ \Rightarrow $ (c). Firstly note that $\Ra(X)=\Ra(A^k)$ implies  $P_{A^k}X=X$. Also, by \eqref{core EP Am}, \eqref{P_Am and Q_Am} and \eqref{projector} we obtain
\[[(A^{\odagger})^m A^m] P_{A^m} P_{A^k}=U\begin{bmatrix}
 I_t & T^{-m}\tilde{T}_m  \\
  0 & 0 \\
  \end{bmatrix}\left[\begin{array}{cc}
I_t & 0 \\
0 & P_{N^m}
\end{array}\right]\left[\begin{array}{cc}
I_t & 0 \\
0 & 0
\end{array}\right]U^*=P_{A^k}.\]
Thus,  $AX^2=(A^{\odagger})^m A^m P_{A^m}P_{A^k}X=P_{A^k} X=X$. \\
(c) $\Rightarrow$ (a) It is easy to see that $AX^2=X$ implies
$X =  A^k X^{k+1}$. So,  $\Ra(X)\subseteq \Ra(A^k)$.  Now, the implication follows from Theorem \ref{characterization 2}.
\end{proof}

It is known that $A^{\core}$ belong to $A\{1\}$. It is of interest to
inquire whether the $m$-weak core inverse belongs to $A\{1\}$ as well.

\begin{theorem} \label{gi} Let $A\in \Cnn$. Then $A^{\core_m}\in A\{1\}$  if and only if $\ind(A) \le 1$.
\end{theorem}
\begin{proof} From \eqref{core EP decomposition} and \eqref{canonical m-weak core} we get
\[AA^{\core_m}A=A \Leftrightarrow
\begin{bmatrix}
    T & ~S+T^{-m}\tilde{T}_m P_{N^m}N \\
    0 & 0
  \end{bmatrix}=\begin{bmatrix}
                   T & S \\
                   0 & N \end{bmatrix} \Leftrightarrow N=0.\]
Thus,  Remark \ref{GM} completes the proof.
\end{proof}

\section{More properties and representations of the $m$-weak core inverse}
In this section we derive some more properties of the $m$-weak core inverse. Also, we investigate the most general representation of this new inverse.

\begin{theorem}\label{properties 3} Let $A\in \Cnn$, $k=\ind(A)$, and $m\in \mathbb{N}$.
Then
\begin{enumerate}[\rm(a)]
\item $AA^{\core_m}A=A^{\weak_{m-1}}A^{m+1}(A^m)^\dag A$.
\item $AA^{\core_m}A^{m+1}= A^{\weak_{m-1}}A^{m+1}$.
\item $A^{\core_m}A^m=A^{\weak_m}A^m$.
\item If $k<2m$, then $A^{\core_m}A^m=A^{d}A^m$.
\item $A^m A^{\core_m}=A^{\odagger}A^mP_{A^m}$.
\item If $A^m$ is EP then $A^m A^{\core_m}=A^{\odagger}A^m$.
\item $A^m A^{\core_m}A^m=A^{\odagger}A^{2m}$
\item $A^{m+1} A^{\core_m}A^m=P_{A^l}A^{2m}$, for $l\ge k$.
\item If $k<2m$ then $A^{m+1} A^{\core_m}A^m=A^{2m}$.
\end{enumerate}
\end{theorem}
\begin{proof}
(a) Immediately follows from Theorem \ref{properties 1} (b) and the fact that $P_{A^m}=A^m(A^m)^\dag$.\\
(b) Follows easily from (a).\\
(c) By definition of the $m$-weak core inverse.\\
(d) By Theorem \ref{properties 2} (e) we know that $A^{\core_m}=(A^d)^{m+1} P_{A^k} A^m P_{A^m}$. Thus, as $k<2m$ we have
\begin{align*}
A^{\core_m}A^m &= (A^d)^{m+1}P_{A^k} A^{2m}\\
&=(A^d)^{m+1}A^{2m}\\
&=A^dA^m.
\end{align*}
(e) From Theorem \ref{properties 1} (a) we know that
$A^{\core_m}= (A^{\odagger})^{m+1} A^m P_{A^m}$. Also, as $A^{\odagger}=A(A^{\odagger})^2$, by applying induction on $m$ it is easy to see
\begin{equation}\label{property core EP}
A^m(A^{\odagger})^{m+1}=A^{m-1}(A^{\odagger})^m=\cdots= A^{\odagger}.
\end{equation}
Thus, the assertion follows from \eqref{property core EP}. \\
(f) By Remark \ref{rem 2} (iii) we know that if $A^m$ is EP then $A^{\core_m}=A^{\weak_m}$. Thus, \eqref{m-WG} implies
\begin{equation}\label{Am by m-weak}
A^m A^{\core_m}=A^m(A^{\odagger})^{m+1}A^m.
\end{equation}
So, \eqref{property core EP} and \eqref{Am by m-weak}  imply  $A^m A^{\core_m}=A^{\odagger}A^m$. \\
(g) Follows directly from (e).\\
(h) By applying (g) we have $A^{m+1} A^{\core_m}A^m=AA^{\odagger}A^m$. Now, the statement follows of the fact that $AA^{\odagger}=P_{A^l}$ for $l\ge k$. \\
(i) From Lemma \ref{projector}, $P_{A^k}=P_{A^l}$ for $l\ge k$. So, as $k<2m$ we have $P_{A^k}A^{2m}=A^{2m}$. Now, the affirmation it is clear from (h).
\end{proof}

Maximal classes of matrices for which general expressions of the $m$-weak core inverse holds, are considered now.

\begin{theorem} Let $A\in \Cnn$ be a matrix of index $k$ written as in (\ref{core EP decomposition}). Let $t=\rank(A^k)$, $G \in \Cnn$, $H \in A^m\{1\}$, and $m\in \mathbb{N}$. Then the following are equivalent:
\begin{enumerate}[\rm(a)]
\item $A^{\core_m}=G A^m H$;
\item $GA^m=A^{\weak_m} A^m$ and $A^{\weak_m}A^m H=A^{\weak_m}P_{A^m}$;
\item $G=A^{\weak_m} + Z (I_n-P_{A^m})$, \,  $H=(A^m)^\dag +(I_n-Q_{A^m}) W$ for arbitrary $Z, W \in {\mathbb C}^{n \times n}$;
\item $G$ and $H$ can be expressed as
\[
G=U \left[
\begin{array}{cc}
T^{-1} & ~T^{-(m+1)}\tilde{T}_mP_{N^m} + Z_{12}(I_{n-t}-P_{N^m}) \\
0 & Z_{22}(I_{n-t}-P_{N^m})
\end{array}\right]U^*,\]
\[
H=U\left[
\begin{array}{cc}
T^{-m}(I_t-M H_{21}) & ~T^{-m}[T^{-m} \tilde{T}_mP_{N^m}-MH_{22}] \\
H_{21} & H_{22}
\end{array}\right]U^*,
\]
where $\tilde{T}_m$ is as in (\ref{core EP Am}), $M:=\tilde{T}_m+T^{-m}\tilde{T}_mN^m$,  $Z_{12}\in \mathbb{C}^{t\times(n-t)}$, $H_{21}\in \mathbb{C}^{(n-t)\times t}$, and $Z_{22},H_{22}\in \mathbb{C}^{(n-t)\times(n-t)}$.
\item $\Ra(GA^m)\subseteq\Ra(A^k)$, $\Nu((A^m)^*)\subseteq \Nu(A^{\weak_m}A^m H)$ and $A^m G A^m=A^{\odagger} A^{2m}$;
\end{enumerate}
\end{theorem}
\begin{proof}
(a) $\Rightarrow$ (b) As  $A^{\core_m}=G A^m H$, from \eqref{m-weak core} and the fact that $A^mHA^m=A^m$ we have
\[GA^m=(GA^mH)A^m=A^{\core_m}A^m=A^{\weak_m}P_{A^m}A^m = A^{\weak_m}A^m,\]
which yields
\[A^{\weak_m}A^mH=GA^mH=A^{\core_m}=A^{\weak_m}P_{A^m}.\]
(b) $\Rightarrow$ (c) According to \cite[p. 52]{BeGr}, it is easy to see that all solutions of equation $GA^m=A^{\weak_m} A^m$ are given by
\[G=A^{\weak_m} + Z (I_n-P_{A^m}), \quad \text{for arbitrary}\quad Z \in {\mathbb C}^{n \times n}.\]
Analogously, the general solution of the equation  $A^{\weak_m}A^m H=A^{\weak_m}P_{A^m}$ is 
\[H=(A^m)^\dag +(I_n-Q_{A^m}) W, \quad \text{ for arbitrary} \quad  W \in \Cnn.
\]
(c) $\Rightarrow$ (a) We note that $A^m=P_{A^m} A^m=A^mQ_{A^m}$. So,
\begin{eqnarray*}
G A^m H &=&  [A^{\weak_m} + Z (I_n-P_{A^m})] A^m \left[(A^m)^\dag +(I_n-Q_{A^m}) W\right] \\
     &=&  A^{\weak_m}  A^m \left[(A^m)^\dag  + (I_n-Q_{A^m}) W\right] \\
     &= & A^{\weak_m}  P_{A^m}  \\
     &=&  A^{\core_m},
\end{eqnarray*}
where the last equality follows from \eqref{m-weak core}.
\\ (a) $\Rightarrow$ (d)
Since $A$ is written as in (\ref{core EP decomposition}), we partition
\[ G =
U \left[
\begin{array}{cc}
G_{11} & G_{12} \\
G_{21} & G_{22}
\end{array}\right]U^* \quad \text{and}\quad H =
U \left[
\begin{array}{cc}
H_{11} & H_{12} \\
H_{21} & H_{22}
\end{array}\right]U^*,\]
according to the sizes of the partition of $A$. \\
Since (a) $\Rightarrow$ (c) is valid, we have 
\begin{equation}\label{G}
G = A^{\weak_m} + Z (I_n-P_{A^m}), \quad \text{ for arbitrary } Z \in \mathbb{C}^n.
\end{equation}
From \eqref{P_Am and Q_Am} for $\ell=m$, we get
\[
P_{A^m}= U\left[\begin{array}{cc}
I_t & 0 \\
0 & P_{N^m}
\end{array}\right]U^*.
\]
Now, by making the following partition with blocks of adequate sizes
$$
Z =
U \left[
\begin{array}{cc}
Z_{11} & Z_{12} \\
Z_{21} & Z_{22}
\end{array}\right]U^*.
$$
From (\ref{canonical m-weak core}) and (\ref{G}) we have
\begin{eqnarray*}
U^* G U
&=& \left[\begin{array}{cc}
T^{-1} & T^{-(m+1)} \tilde{T}_m P_{N^m} \\
0 & 0
\end{array}\right] + \left[
\begin{array}{cc}
Z_{11} & Z_{12} \\
Z_{21} & Z_{22}
\end{array}\right]
\left[\begin{array}{cc}
0 & 0 \\
0 & I_{n-t}-P_{N^m}
\end{array}\right]\\
& = &
\left[
\begin{array}{cc}
T^{-1} & T^{-(m+1)} \tilde{T}_mP_{N^m} + Z_{12}(I_{n-t}-P_{N^m}) \\
0 & Z_{22}(I_{n-t}-P_{N^m})
\end{array}\right].
\end{eqnarray*}
On the other hand, since (a) $\Rightarrow$ (b) has been proved, we can deduce that
$A^{\weak_m}A^m H=A^{\weak_m}P_{A^m}$ holds.
In consequence, by comparing their blocks
we arrive at
\[T^{m-1} H_{11}+(T^{-1}\tilde{T}_m+T^{-(m+1)}\tilde{T}_m N^m) H_{21}=T^{-1}\]
 and
\[T^{m-1} H_{12}+(T^{-1}\tilde{T}_m+T^{-(m+1)}\tilde{T}_m N^m) H_{22}=T^{-(m+1)}\tilde{T}_m P_{N^m}.\] Thus, (d) is derived.
\\ (d) $\Rightarrow$ (a)  By taking $Z_{12}=0$, $H_{21}=0$,  and $Z_{22}=H_{22}=0$, an easy computation shows  \begin{eqnarray*}
G A^m H &=& U\left[
\begin{array}{cc}
T^{-1} & T^{-(m+1)} \tilde{T}_mP_{N^m} \\
0 & 0
\end{array}\right]\left[
\begin{array}{cc}
T^m &  \tilde{T}_m  \\
0 & N^m
\end{array}\right]\left[
\begin{array}{cc}
T^{-m} & T^{-(2m+1)} \tilde{T}_mP_{N^m}  \\
0 & 0
\end{array}\right]U^*\\
&=& U\left[
\begin{array}{cc}
T^{-1} & T^{-(m+1)} \tilde{T}_mP_{N^m} \\
0 & 0
\end{array}\right]\\
&=& A^{\core_m}.
\end{eqnarray*}
(b) $\Rightarrow$ (e) Clearly, $GA^m=A^{\weak_m} A^m$ yields $\Ra(GA^m)\subseteq \Ra(A^{\weak_m})=\Ra(A^k)$. Moreover, as $A^m(A^{\odagger})^{m+1}=A^{\odagger}$, the condition $GA^m=A^{\weak_m} A^m$ also implies  \[A^mGA^m=A^m A^{\weak_m} A^m A^m=A^m(A^{\odagger})^{m+1}A^{2m}=A^{\odagger}A^{2m}.\]
On the other hand, as $A^{\weak_m}A^m H=A^{\weak_m}P_{A^m}$  we get
\[\Nu((A^m)^*)=\Nu((A^m)^\dag)\subseteq\Nu(P_{A^m})\subseteq \Nu(A^{\weak_m}P_{A^m})=\Nu(A^{\weak_m}A^m H).\]
(e) $\Rightarrow$ (b) By $\Ra(GA^m)\subseteq\Ra(A^k)$ we have $GA^m=A^kQ$ for some matrix $Q$. We know that $A^{\weak_m}A^{k+1}=A^k$ and $A^{\weak_m}=(A^{\odagger})^{m+1}A^m$. In consequence, the hypothesis $A^mGA^m=A^{\odagger}A^{2m}$ implies
\begin{eqnarray*}
GA^m &=& A^{\weak_m}A^{k+1} Q \\
&=& (A^{\odagger})^{m+1}A A^m G A^m \\
&=& (A^{\odagger})^{m+1}A A^{\odagger} A^{2m}.
\end{eqnarray*}
Now, by Theorem \ref{properties 3} (g) we know that $A^m A^{\core_m}A^m=A^{\odagger}A^{2m}$. Also, from Theorem \ref{properties 2} (a) we have that $A^{\core_m}AA^{\core_m}=A^{\core_m}$. Therefore,
\begin{eqnarray*}
GA^m &=& (A^{\odagger})^{m+1}A [A^{\odagger} A^{2m}] \\
&=& (A^{\odagger})^{m+1}A A^m A^{\core_m}A^m \\
&=& [(A^{\odagger})^{m+1} A^m] A A^{\core_m}A^m \\
&=& A^{\core_m}AA^{\core_m} A^m \\
&=& A^{\core_m}A^m.
\end{eqnarray*}
It is clear that $\Nu((A^m)^*)=\Ra(I_n-P_{A^m})$. So, $\Nu((A^m)^*)\subseteq \Nu(A^{\weak_m}A^m H)$ is equivalent to
$\Ra(I_n-P_{A^m})\subseteq \Nu(A^{\weak_m}A^m H)$ which in turn is equivalent to \[A^{\weak_m}A^m H (I_n-P_{A^m})=0.\] In consequence, as $A^mHA^m=A^m$ we obtain \[A^{\weak_m}A^m H=A^{\weak_m}A^m H A^m (A^m)^\dag=A^{\weak_m}P_{A^m}.\] This completes the implication.
\end{proof}

\section{The $m$-weak core inverse and the Hartwig-Spindelb\"ock decomposition}
In this section we get an alternative way of how to compute the $m$-weak core inverse of $A$.

For any matrix $A\in \Cnn$ of rank $r>0$, the Hartwig-Spindelb\"ock decomposition is given by
\begin{equation} \label{HS}
 A = U\left[\begin{array}{cc}
\Sigma K & \Sigma L \\
0 & 0
\end{array}\right]U^*,
\end{equation}
where $U\in \Cnn$ is unitary, $\Sigma=\text{diag} (\sigma_1 I_{r_1},\sigma_2 I_{r_2},\dots, \sigma_t I_{r_t})$ is a diagonal matrix, the diagonal entries $\sigma_i$ being singular values of $A$, $\sigma_1>\sigma_2>\cdots>\sigma_t>0,$
$r_1+r_2+\cdots+r_t=r$, and $K\in \mathbb{C}^{r\times r}$, $L\in \mathbb{C}^{r\times(n-r)}$ satisfy
\begin{equation}\label{KK*+LL*}
KK^*+LL^*=I_r.
\end{equation}
If $A$ is of the form (\ref{HS}), from \cite[Formula (10)]{FeLeTh1} we know that its core-EP inverse  is as follows
\begin{equation}\label{core EP canonical hartwig}
A^{\odagger}=U\left[\begin{array}{cc}
(\Sigma K)^{\odagger}  & 0 \\
0 & 0
\end{array}\right]U^*.
\end{equation}

Next, we present another representation of the $m$-weak core inverse by using the Hartwig-Spindelböck decomposition.

\begin{theorem} \label{canonical m-weak hartwig}
Let $A\in \Cnn$ be a matrix  written as in (\ref{HS}).
Then
\begin{equation}\label{hartwig m-weak core}
A^{\core_m}=U\left[\begin{array}{cc}
(\Sigma K)^{\weak_m} P_{(\Sigma K)^{m-1}} & 0 \\
0 & 0
\end{array}\right]U^*.\end{equation}
\end{theorem}
\begin{proof} By Theorem \ref{properties 1} (a) we know that  $A^{\core_m}=(A^{\odagger})^{m+1} A^m P_{A^m}$.  \\
From \eqref{HS} we have
\begin{equation}\label{A^m hartwig}
 A^m = U\left[\begin{array}{cc}
(\Sigma K)^m & (\Sigma K)^{m-1} \Sigma L \\
0 & 0
\end{array}\right]U^*.
\end{equation}
In consequence, we obtain
\[
 (A^m)^{\dagger} = U\left[\begin{array}{cc}
P^* R^{\dagger} & 0 \\
Q^* R^{\dagger} & 0
\end{array}\right]U^*,
\]
where $R=PP^*+QQ^*$, $P=(\Sigma K)^m$ and $Q=(\Sigma K)^{m-1} \Sigma L.$
This implies that
\begin{equation} \label{P_Am}
 P_{A^m}=A^m(A^m)^{\dagger} = U\left[\begin{array}{cc}
P_R & 0 \\
0 & 0
\end{array}\right]U^*.
\end{equation}
As in the proof of \cite[Theorem 3.2]{FeLeTh1}, an alternative expression for $P_R$ is given by
\begin{equation} \label{P_R}
  P_R=P_{(\Sigma K)^{m-1}}.
\end{equation}
Now, (\ref{P_Am}) and (\ref{P_R}) imply
\begin{equation} \label{P_Am 2}
 P_{A^m}=
U\left[\begin{array}{cc}
P_{(\Sigma K)^{m-1}} & 0 \\
0 & 0
\end{array}\right]U^*.
\end{equation}
So, from \eqref{core EP canonical hartwig}, \eqref{A^m hartwig}, and the fact that $A^{\core_m}=(A^{\odagger})^{m+1} A^m P_{A^m}$, we get
\[A^{\core_m}=U\left[\begin{array}{cc}
[(\Sigma K)^{\odagger}]^{m+1} (\Sigma K)^m P_{(\Sigma K)^{m-1}} & 0 \\
0 & 0
\end{array}\right]U^*.\]
Hence, from \eqref{m-WG} we deduce \eqref{hartwig m-weak core}.
\end{proof}

The previous result allows us to recover  representations of the core-EP and WC inverses obtained in \cite[Theorem 3.2]{FeLeTh1} and \cite[Theorem 3.18]{FeLePrTh}, respectively.

\begin{corollary} \label{canonical}
Let $A\in \Cnn$ be a matrix  written as in (\ref{HS}).
\begin{enumerate}[(a)]
\item If $m=1$ then
\begin{equation*}
 A^{\core_m}=A^{\weak,\dag}= U\left[\begin{array}{cc}
(\Sigma K)^{\weak} & 0 \\
0 & 0
\end{array}\right]U^*.
\end{equation*}
\item If $m\ge k$ then
\begin{equation*}
 A^{\core_m}=A^{\odagger}= U\left[\begin{array}{cc}
(\Sigma K)^{\odagger} & 0 \\
0 & 0
\end{array}\right]U^*.
\end{equation*}
\end{enumerate}
\end{corollary}
\begin{proof}
(a)
Follows from Remark \ref{rem 1}. \\
(b) Since $m\ge k$,  by Remark \ref{rem 1} we have $(\Sigma K)^{\weak_m}=(\Sigma K)^d$. Moreover, as $A$ has index $k$ we know that $\Sigma K$ has index $k-1$ \cite[Lemma 2.8]{MaTh}. So, from Lemma \ref{projector} we get $P_{(\Sigma K)^{m-1}}=P_{(\Sigma K)^{k-1}}$ in our case. Thus, from \eqref{core-EP inverse} we obtain $(\Sigma K)^{\odagger}=(\Sigma K)^d P_{(\Sigma K)^{k-1}}$. Now,   the affirmation follows from \eqref{hartwig m-weak core}. 
\end{proof}

\section{Relation of $m$-weak core inverse with other generalized inverses}

In this section we present necessary and sufficient conditions for which the $m$-weak core inverse inverse coincides  with certain known generalized inverses  by using core-EP decomposition.

\begin{theorem} \label{thm equalities}
Let $A\in \Cnn$ be a matrix with $\text{Ind}(A)=k$ written as in (\ref{core EP decomposition}). Then
\begin{enumerate}[{\rm(a)}]
\item  $A^{\core_m}=A^\dag$ if and only if $S=0$ and $N=0$;
\item $A^{\core_m}=A^d$ if and only if  $\tilde{T}_m P_{N^m}=T^{m-k}\tilde{T}_k$;
\item  $A^{\core_m}=A^{\odagger}$ if and only if $\tilde{T}_m N^m=0$;
\item $A^{\core_m}=A^{d,\dag}$ if and only if $\tilde{T}_m P_{N^m}=T^{m-k}\tilde{T}_k P_N$;
\item $A^{\core_m}=A^{\weak,\dag}$ if and only if $\tilde{T}_m P_{N^m}=T^{m-1}S P_N$;
\item  $A^{\core_m}=A^{\weak_m}$ if and only if $\Nu((N^m)^*)\subseteq \Nu(\tilde{T}_m)$.
\end{enumerate}
\end{theorem}

\begin{proof} (a) From \eqref{MP of Am} for $m=1$,  the Moore-Penrose inverse of $A$ is
 \begin{equation} \label{MP of A}
A^\dagger = U\left[\begin{array}{cc}
T^*\Delta_1 & -T^*\Delta_1 S N^\dagger\\
\Omega_1^*\Delta_1 & N^\dagger-\Omega_1^*\Delta_1 S N^\dagger
\end{array}\right]U^*.
\end{equation}
From \eqref{canonical m-weak core} and \eqref{MP of A} we have  that  $A^{\core_m}=A^\dag$ if and only if the following conditions simultaneously hold:
\begin{enumerate}[(i)]
\item $T^{-1}= T^* \Delta_1 $,
\item $T^{-(m+1)}\tilde{T}_m P_{N^m}=-T^* \Delta_1 S N^\dag$,
\item $0=\Omega_1^* \Delta_1$,
\item $0=N^\dag-\Omega_1S^* \Delta_1 S N^\dag$.
\end{enumerate}
We will prove that (i)-(iv) hold if and only $S=0$ and $N=0$. In fact, as $\Delta_1$ is nonsingular, from (iii) we obtain $\Omega_1=0$. Thus, (iv) implies   $N^\dag=0$, whence $N=0$. Therefore, as $\Omega_1=(I_{n-t}-Q_n)S=0$ we have $S=0$.  Conversely, if $S=0$ and $N=0$, clearly (i)-(iv) are true. \\
(b) By \cite[Theorem 3.9]{FeLeTh2} the Drazin inverse of $A$ is
\begin{equation} \label{Drazin representation}
A^d = U\left[\begin{array}{cc}
T^{-1} & T^{-(k+1)}\tilde{T}_k \\
0 & 0
\end{array}\right]U^*.
\end{equation}
From \eqref{canonical m-weak core} and \eqref{Drazin representation} we have that   $A^{\core_m}=A^d$ is equivalent to
$T^{-(m+1)}\tilde{T}_m P_{N^m}=T^{-(k+1)}\tilde{T}_k$, whence $\tilde{T}_m P_{N^m}=T^{m-k}\tilde{T}_k$.
\\(c) From \eqref{canonical m-weak core} and \eqref{core EP representation},
$A^{\core_m}=A^{\odagger}$ if and only if $T^{-(m+1)}\tilde{T}_m P_{N^m}=0$  which is clearly equivalent to $\tilde{T}_m N^m=0$.
\\ (d) By \cite[Theorem 3.11]{FeLeTh2} we know that
\begin{equation}\label{DMP representation}
A^{d,\dag} = U\left[\begin{array}{cc}T^{-1} & T^{-(k+1)}\tilde{T}_k P_N \\0 & 0\end{array}\right]U^*.
 \end{equation}
Now, the affirmation follows directly from \eqref{canonical m-weak core} and \eqref{DMP representation}. \\
(e) By \cite[Theorem 3.12]{FeLePrTh} we know that
\begin{equation}\label{WC representation}
A^{\weak} = U\left[\begin{array}{cc}
T^{-1} & T^{-2} S P_N\\
0 & 0
\end{array}\right]U^*.
\end{equation}
In consequence, by comparing the blocks in
 \eqref{canonical m-weak core} and \eqref{WC representation} we arrive at  (e). \\
(f) From \eqref{m-WG}, \eqref{core EP representation} and \eqref{core EP Am}, the $m$-weak group inverse inverse of $A$ can be expressed as
\begin{equation} \label{m-WG representation}
A^{\weak_m} = U\left[\begin{array}{cc}
T^{-1} & T^{-(m+1)}\tilde{T}_m \\
0 & 0
\end{array}\right]U^*.
\end{equation}
Thus,  from \eqref{canonical m-weak core} and \eqref{m-WG representation}  we have that  $A^{\core_m}=A^{\weak_m}$ is equivalent to $T^{-(m+1)}\tilde{T}_m P_{N^m}=T^{-(m+1)}\tilde{T}_m$ which is equivalent to $\tilde{T}_m (I_{n-t}-P_{N^m})=0$ because $T$ is nonsingular. Clearly, $\tilde{T}_m (I_{n-t}-P_{N^m})=0$ if and only if $\Ra(I_{n-t}-P_{N^m})\subseteq \Nu(\tilde{T}_m)$ which in turn is equivalent to  $\Nu((N^m)^*)\subseteq \Nu(\tilde{T}_m)$.
\end{proof}

We conclude this section with some necessary and sufficient conditions under which the $m$-weak core inverse of $A$ coincides with different transformations of  $A$.
Before, recall $A$ is a tripotent matrix if $A^3=A$ while  $A$ is a partial isometry if $AA^*A=A$ (or $A^\dag=A^*$).

\begin{theorem} Let $A\in \Cnn$ with $k=\ind(A)$. Then
 \begin{enumerate}[\rm(a)]
\item $A^{\core_m} =0$ if and only if $A^k=0$.
\item $A^{\core_m} =A$ if and only if $A$ is an EP tripotent matrix.
\item $A^{\core_m}=A^*$ if and only if $A$ is an EP partial isometry.
\item $A^{\core_m}=P_A$ if and only if $A$ is an idempotent matrix.
\end{enumerate}
\end{theorem}
\begin{proof}
(a) By applying Theorem \ref{properties 2} (g) we have $\ra(A^{\core_m}) = \ra(A^k)=0$, whence $A^k=0$. \\
In order to prove statements (b)-(d), let us consider  $A$ as in \eqref{core EP decomposition}. \\
(b) From \eqref{canonical m-weak core} we have
\begin{eqnarray*}
A^{\core_m} =A  &\Leftrightarrow &  ~~ T^{-1}= T,~~T^{-(m+1)}
\tilde{T}_mP_{N^m}=S~~\text{and}~~N=0;
\\  &\Leftrightarrow & T^2=I_t, ~~ S=0 ~~\text{and}~~ N=0;
\\ &\Leftrightarrow & A ~~\text{is an EP tripotent matrix}.
\end{eqnarray*}
(c) From \eqref{canonical m-weak core} we get
\begin{eqnarray*}
A^{\core_m} =A^*  &\Leftrightarrow &  ~~ T^{-1}= T^*,~~T^{-(m+1)}
\tilde{T}_mP_{N^m}=0,~~0=S^*~~\text{and}~~0=N^*;
\\  &\Leftrightarrow & TT^*=I_t, ~~ S=0 ~~\text{and}~~ N=0;
\\ &\Leftrightarrow & A ~~\text{is an EP partial isometry}.
\end{eqnarray*}
(d) By \eqref{canonical m-weak core} and \eqref{P_Am and Q_Am} for $m=1$, we obtain
 \begin{eqnarray*}
A^{\core_m} =P_A  &\Leftrightarrow &  ~~ T^{-1}= I_t,~~T^{-(m+1)}
\tilde{T}_mP_{N^m}=0,~~\text{and}~~0=P_N;
\\ &\Leftrightarrow &  ~~ T^{-1}= I_t,~~T^{-(m+1)}
\tilde{T}_mP_{N^m}=0,~~\text{and}~~0=N;
\\  &\Leftrightarrow & T=I_t ~~\text{and}~~ N=0;
\\ &\Leftrightarrow & A~~\text{is an idempotent matrix}.
\end{eqnarray*}
\end{proof}

\section*{Declarations} 

\noindent {\bf  Funding} 
 
\noindent This work was supported by the Universidad Nacional de R\'{\i}o Cuarto under Grant PPI 18/C559;
CONICET under Grant PIP 112-202001-00694CO;  CONICET under Grant PIBAA 28720210100658CO; and Universidad Nacional de La Pampa, Facultad de Ingenier\'ia under Grant 135/19.

\noindent {\bf Disclosure statement}
 
\noindent The authors report there are no competing interests to declare.

\end{document}